\documentclass[12pt]{amsart}
\usepackage{latexsym, amsmath, amscd, amssymb, amsthm} 
\newtheorem{te}{Theorem}[section]
 \newtheorem{oz}{Definition}[section]

\begin{document}

\noindent

 \title[]{ Covariants of binary forms and new identities for Bernoulli, Euler and Hermite polynomials }

\author{Leonid Bedratyuk}\address{Khmelnitskiy national university, Insituts'ka, 11,  Khmelnitskiy, 29016, Ukraine}

\begin{abstract} Using the methods of classical invariant theory a general approach to finding  of identities for  Bernulli, Euler and Hermite polynomials  is proposed.
 
\end{abstract}
\maketitle

\section{Introduction}
The relationship between the  group representations theory  and special functions is well known, see. \cite{Vil}.
 In this paper we establish the relationship between the classical invariant theory  and identities for  Bernulli, Euler, and Hermite polynomials.

The polynomials of Bernoulli    $B_n(x),$ Euler $E_n(x)$ and Hermite  $H_n(x),$ $n=0,1,2,\ldots $  are defined by the following generating functions
\begin{gather*}
\frac{t e^{xt}}{e^t-1}=\sum_{i=0}^{\infty} B_n(x) \frac{ t^n}{n!}, \frac{2e^{xt}}{e^t+1}=\sum_{i=0}^{\infty}E_n(x)  \frac{t^n}{n!}, e^{xt-\frac{t^2}{2}}=\sum_{i=0}^{\infty}H_n(x) \frac{ t^n}{n!}. 
\end{gather*}
Particularly $B_0(x)=E_0(x)=H_0(x)=1.$ The numbers  $B_n=B_n(0)$   are called the Bernoulli numbers   and the numbers $E_n=E_n(0)$ are called the Euler numbers.
All these types of polynomials are special cases of the Appell polynomials $\mathcal{A}=\{A_n(x) \},$ where $\deg(A_n(x))=n$ and the polynomials  satisfy the identity
\begin{gather}
A'_n(x)=n A_{n-1}(x), n=0,1,2,\ldots .
\end{gather}
It is clear that the polynomials $ \{x ^ n \} $ are  the  Appell polynomials also.
Denote by  $\mathcal{B},$ $\mathcal{E},$ $\mathcal{H},$ the Bernulli, Euler and Hermite polynomials, respectively. Also, put   $\mathcal{T}:=\{ 1, x, x^2, \ldots, x^n, \ldots. \}.$

We are interested in finding all polynomial identities for the Appell polynomials, i.e.  identities  of the form
$$
F(A_0(x),A_1(x),\ldots, A_n(x))=0,
$$
where  $F$ is some polynomial of  $n+1$ variables.
 
First, consider the motivating examples.
Let
\begin{gather*}
\Delta(x):=
\begin{vmatrix}
B_0(x) & B_1(x) \\
B_1(x) & B_2(x) \\
\end{vmatrix}=B_0(x) B_2(x)-B_1(x)^2.
\end{gather*}
Taking into account  $(1)$ we have
\begin{gather*}
\Delta(x)'=(B_0(x) B_2(x)-B_1(x)^2)'=B_0(x)' B_2(x)+B_0(x) B_2(x)'-2 B_0(x) B_1(x)=\\
=2 B_0(x) B_1(x)-2 B_0(x) B_1(x)=0
\end{gather*}
Thus $\Delta (x)$ is a constant, and it is clear that this constant is equal to $ \Delta (0).$
Substituting the corresponding Bernoulli polynomial, we find the constant and get the following identity 
 \begin{gather*}
B_0(x) B_2(x)-B_1(x)^2=B_0 B_2-B_1^2=-\frac{1}{12}.
\end{gather*}
Similarly, $
E_0(x) E_2(x)-E_1(x)^2=-\frac{1}{4} 
$
and  $
H_0(x) H_2(x)-H_1(x)^2=-1.
$

Consider now  the following differential operator 
 \begin{gather*}
\mathcal{D}:=a_0 \frac{\partial}{\partial a_1}+2 a_1 \frac{\partial}{\partial a_2}+\cdots+n a_{n-1} \frac{\partial}{\partial a_n},
\end{gather*}
which acts on  polynomials of the variables $a_0, a_1, \ldots, a_n.$ 
The action is very similar to (1).
Also, it is easy to see that  $\mathcal{D}(a_0 a_2-a_1^2)=0.$

Consider the polynomial
$$
\Delta_3: =\begin{vmatrix}
a_0 & 3 a_1 &3 a_2 & a_3 &0 \\
0 & a_0 & 3 a_1 &3 a_2 & a_3 \\
3a_0 & 6 a_1 &3 a_2 & 0 &0 \\
0&3a_0 & 6 a_1 &3 a_2 & 0  \\
0&0&3a_0 & 6 a_1 &3 a_2  \\
\end{vmatrix}
$$
Note that  the polynomial  $\Delta_3$ is, up to a factor,  the discriminant of an binary form of degree 3:
$$
a_{{0}}{X}^{3}+3\,a_{{1}}{X}^{2}Y+3\,a_{{2}}X{Y}^{2}+a_{{3}}{Y}^{3}.
$$
By applying the determinant derivative rule we get that  $\mathcal{D}(\Delta)$ equals to sum of  5 determinants  each of them equal to zero. Thus $\mathcal{D}(\Delta)=0.$ Similarly, for the determinant 
$$
\Delta_3(\mathcal{A}): =\begin{vmatrix}
A_0(x) & 3 A_1(x) &3 A_2(x) & A_3(x) &0 \\
0 & A_0(x) & 3 A_1(x) &3 A_2(x) & A_3(x) \\
3A_0(x) & 6 A_1(x) &3 A_2(x) & 0 &0 \\
0&3A_0(x) & 6 A_1(x) &3 A_2(x) & 0  \\
0&0&3A_0(x) & 6 A_1(x) &3 A_2(x)  \\
\end{vmatrix}
$$
we obtain   $\Delta(\mathcal{A})'=0.$  Therefore, for any  Appell polynomials $\{A_n(x)\}$  the identity  $$\Delta_3(\mathcal{A})={\rm const}$$ holds. By direct calculations we obtain
$$
\Delta_3(\mathcal{B})=\frac{1}{16}, \Delta_3(\mathcal{E})=\frac{27}{16}, \Delta_3(\mathcal{H})=108.
$$

These examples lead to the hypothesis that if a polynomial $S(a_0, a_1,\ldots, a_n)$ satisfies the condition $\mathcal{D}(S(a_0, a_1,\ldots, a_n))=0,$ 
then the polynomial  $S(A_0(x), A_1(x),\ldots, A_n(x))$ is a constant, thus determines the identities between the Appell polynomials. 

We proceed, by way of summarizing the approaches, as follows. 
Let   $\mathbb{K}[a_0,a_1,\ldots, a_n]$ and  $\mathbb{K}[x]$ be the algebras of polynomials over a field $\mathbb{K}$ of characteristic zero. 
Let us consider the substitution homomorphism
 $\varphi_{\mathcal{A}}:\mathbb{K}[a_0,a_1,\ldots, a_n] \to \mathbb{K}[x] $ defined  by   $\varphi_{\mathcal{A}}(a_i)=A_i(x).$  Put
$$
\ker^* \varphi_{\mathcal{A}}:=\{ S \in  \mathbb{K}[ a_0,a_1,...,a_n]  \mid \varphi_{\mathcal{A}}(S) \in \mathbb{K} \}.
$$ We will prove that any element  $S(a_0,a_1,\ldots,a_n)$ of the subalgebra $ \ker^* \varphi_{\mathcal{A}}$ defines an identity 
$$
S(\mathcal{A})=\| S(\mathcal{A}) \|,
$$
here  $S(\mathcal{A}):=S(A_0(x),A_1(x),\ldots, A_n(x))$ and $\| S(\mathcal{A}) \|=S(A_0(0),A_1(0),\ldots, A_n(0)).$

Therefore, the problem of describing all polynomial identities for the  Appell polynomials is reduced to that of describing  the algebra $\ker^* \varphi_{\mathcal{A}}.$
It will be shown in the paper that the algebra $\ker^* \varphi_{\mathcal{A}}$ is isomorphic to the algebra of covariants of binary form of order $ n. $

This idea can also be applied  for finding identities for  different types of  Appell polynomials. For instance, we have:
 
\begin{gather*}
\begin{vmatrix}
B_0(x) & E_0(x)&H_0(x) \\
B_1(x) & E_1(x)&H_1(x) \\
B_2(x) & E_2(x)&H_2(x) \\
\end{vmatrix}=\\ =B_{{0}}(x) E_{{1}}(x) H_{{2}}(x) -B_{{0}}(x) H_{{1}}(x) E_{{2}}(x)-B_{{1}}(x) E_{{0}}(x) H_{{2}}(x)+\\ +B_{{
1}}(x) H_{{0}}(x) E_{{2}}(x) +B_{{2}}(x) E_{{0}}(x)H_{{1}}(x)-B_{{2}}(x)H_{{0}}(x)E_{{1}}(x)=\frac{1}{12}.
\end{gather*}
The problem of describing  all such polynomials identities for different  Appell polynomials is reduced to that of describing of the algebra of joint covariants for several binary forms. The algebra of covariants  of binary form and the algebra of  joint covariants for several  binary forms were an
object of research in the classical invariant theory of the 19th century. In particular, developed efficient methods to find elements of this algebra.

We will deal mainly with the algebra of semi-invariants rather than the algebra covariants. These algebras are isomorphic.  But,   the algebra  of semi-invariants is simpler object for computation.

In this  paper  we give a brief introduction to the theory of covariants and semi-invariants of binary form  on the locally nilpotent derivations language.
Based on the  classical invariant theory approach the  several types of identities for Appell polynomials are constructed.


\section{Covariants and semi-invariants of binary forms}

Let us recall, that a derivation of a ring   $R$ is an additive  map  $D$ satisfying the Leibniz rule: 
$$
D(r_1 \, r_2)=D(r_1) r_2+r_1 D(r_2), \text{  äëÿ âñ³õ }  r_1, r_2 \in R.
$$
A derivation $D$ of a ring $R$ is called locally nilpotent if for every $r \in R$ there is an $n \in \mathbb{N}$ such that $D^n(r)=0.$
The subring 
$$
\ker D:=\left \{ f \in R|  D(f)=0 \right \},
$$
is called the kernel of the derivation $D.$

Let us consider the algebra of polynomials   $\mathbb{K}[a_0,a_1,\ldots, a_n]$ over the field $\mathbb{K}$ of characteristic 0. Define the derivations $\mathcal{D},$ $\mathcal{D}^*$  and  $E$  of the algebra  $\mathbb{K}[A_n]$  by 
$$
\mathcal{D}(a_i)=i a_{i-1}, \mathcal{D}^*(a_i)=(n-i) a_{i+1}, E(a_i)=(n-2i) a_i.
$$
Note, that the derivations  $\mathcal{D},$ $\mathcal{D}^*,$    $E$  define  a representation  of the Lie algebra $\mathfrak{sl}_2(\mathbb{K}).$

Consider  the derivations $\mathcal{D}-Y \frac{\partial }{ \partial X}$ and $\mathcal{D}^*-X \frac{\partial}{\partial Y}$ of the polynomial algebra 
 $\mathbb{K}[a_0,\ldots, a_n, X,Y].$ 
Let us recall some concepts of  the classical invariant theory.
\begin{oz}
The homogeneous polynomial
\begin{gather}
\alpha(X,Y):=a_0X^n+n a_1 X^{n-1} Y+\cdots+{n \choose i} a_i X^{n-i} Y^i+\cdots + a_n Y^n. 
\end{gather}
is called the generic binary form of order  $n.$
 \end{oz}
\begin{oz}  
$$
\begin{array}{ll}
{\bf 1.}  & \text {The algebra  } \displaystyle  \mathcal{C}_n:=\ker( \mathcal{D}-Y \frac{\partial }{ \partial X})  \bigcap \ker (\mathcal{D}^*-X \frac{\partial}{\partial Y}) \text { is called  }\\
&\\
&  \text { the algebra of covariants for the generic binary form (2)};\\
{\bf 2.}  & \text {The algebra  } \displaystyle  \mathcal{S}_n:=\ker( \mathcal{D}) \text { is called  the algebra of  semi-invariants }\\
&  \text { for the generic binary form};\\

{\bf 3.}   & \text {The algebra  } \displaystyle \mathcal{I}_n:=\ker( \mathcal{D})  \bigcap \ker (\mathcal{D}^*) \text { is called   }\\
&  \text { the algebra of  invariants for the generic binary form};\\
\end{array}
$$
\end{oz}
The elements of the algebras  $\mathcal{C}_n,$ $\mathcal{S}_n,$  $\mathcal{I}_n$ are called the covariants, semi-invariants and invariants of the binary form. The following   obvious inclusion holds: $\mathcal{I}_n \subset \mathcal{C}_n$ i $\mathcal{I}_n \subset \mathcal{S}_n.$ It is well know that these algebras are  finitely generated algebras.

{\bf Example 2.1} It is easy to check  that the generic form $\alpha(X,Y)$ is a covariant and its leading coefficient  $a_0$ (in the ordering $X>Y$)  is a semi-invariant for the binary form of order $n.$ Also, the element $a_0 a_2-a_1^2$ is an invariant of the binary form of degree $2.$ 

Let  $$\varkappa:\mathcal{C}_n \to \mathcal{S}_n,$$  
be the $\mathbb{K}$-linear map that takes each homogeneous covariant  to its leading coefficient.

The following theorem holds
\begin{te}[\cite{Rob},\cite{Olver}]
The map  $\varkappa$ is the homomorphism of the algebras  $\mathcal{C}_n$ and $\mathcal{S}_n.$
\end{te}
The inverse map  $\varkappa^{-1}:\mathcal{S}_n \to \mathcal{C}_n,$  can be defined   as follows:
$$
\varkappa^{-1}(s)=\sum_{i=0}^{{\rm ord}(s)} \frac{( \mathcal{D}^*)^i(s)}{i!} X^{{\rm ord}(s)-i}Y^i,
$$
where 
 $${\rm ord}(s)=\max \{ k|( \mathcal{D}^*)^k(s) \neq 0 \}.$$
The integer number ${\rm ord}(s)$ is called the order of the semi-invariant  $s.$ The degree of covariant with respect  to the variables $X, Y $ is called the order of the covariant.

Similarly, we can define the algebras of covariants, semi-invariants and invariants of several generic binary forms.
Let us consider  the following three  generic binary forms of order $n:$
\begin{gather*}
\beta(X,Y):=b_0X^n+nb_1 X^{n-1} Y+\cdots+{m \choose i} b_i X^{n-i} Y^i+\cdots + b_n Y^n,\\
\gamma(X,Y):=c_0X^n+n c_1 X^{n-1} Y+\cdots+{m \choose i} c_i X^{n-i} Y^i+\cdots + c_n Y^n,\\
\delta(X,Y):=d_0 X^n+n d_1 X^{n-1} Y+\cdots+{n \choose i} d_i X^{n-i} Y^i+\cdots + d_n Y^n,
\end{gather*}
Extend the derivations
 $\mathcal{D},$ $\mathcal{D}^*$ to the polynomial algebra 
$$\mathbb{K}[a_0,\ldots,a_n,b_0,\ldots,b_n,c_0,\ldots,c_n,d_0,\ldots,d_n],$$ by  $\mathcal{D}(b_i)=i b_{i-1},$  $\mathcal{D}^*(b_i)=(n-i) b_{i+1},$ $\mathcal{D}(c_i)=i c_{i-1},$  $\mathcal{D}^*(c_i)=(n-i) c_{i+1}$   and $\mathcal{D}(d_i)=i d_{i-1},$  $\mathcal{D}^*(d_i)=(n-i) d_{i+1}.$

Then the subalgebra $\ker( \mathcal{D}-Y  \frac{\partial}{\partial X})  \bigcap \ker (\mathcal{D}^*-X \frac{\partial}{\partial Y})$ of $\mathbb{K}[a_0,\ldots, d_n,X,Y]$ is called the algebra of joint covariants of the forms  $\alpha(X,Y),$ $\beta(X,Y),$ $\gamma(X,Y)$ and  $\delta(X,Y).$ The algebras of joint semi-invariants and joint invariants 
can similarly be defined.

The main computational tool of the classical invariant theory is the transvectant.
\begin{oz}
 The $r$-th transvectant of two covariants $f,$ $g$ of orders  $n$ and  $m$ is called the following covariant  
$$
(f,g)^r=\sum_{i=0}^r (-1)^i { r \choose i } \frac{\partial^r f}{\partial X^{r-i} \partial Y^i}   \frac{\partial^r g}{\partial X^{i} \partial Y^{r-i}}, r\leq \min(n,m).
$$
\end{oz}
For instance, the transvectants  $(f,g)^1$ ³ $(f,f)^2$ are equal  to the Jacobian  $J(f,g)$ and to the  Hessian  ${\rm Hes}(f).$

It is well known, see   \cite{Olver}, that each covariant can be represented by transvectants.

Computationally, the semi-invariants is  much simple objects that the covariants. To generate semi-invariants in  \cite{Bed_In7} we introduced the semi-transvectant as an analogue of the transvectant.
\begin{oz}
The semi-invariant   
$$
[p,q]^r:=\varkappa \left((\varkappa^{-1}(p),\varkappa^{-1}(q))^r\right), r \leqslant \min({\rm  ord}(p), {\rm  ord}(q)).
$$
is called the $r$-th semi-transvectant  of the semi-invariants  $p,q \in  \mathbb{K}[a_0,\ldots,d_n]$
\end{oz}
The formula holds

\begin{gather}
[p,q]^r=\sum_{i=0}^r (-1)^i { r \choose i } \frac{(\mathcal{D}^*)^i(p)}{[{\rm  ord}(p)]_i}   \frac{(\mathcal{D}^*)^{r-i}(q)}{[{\rm  ord}(q)]_{r-i}},\label{main}
\end{gather}
where  $[m]_i=m(m-1)\ldots (m-(i-1))$ is the falling factorial.

Directly from the definition we get  the following properties
$$
\begin{array}{l}
\displaystyle [p,q]^0=p\, q,\\ \\
\displaystyle [f,g]^k=(-1)^k [g,f]^k, \text{ so }  [f,f]^k=0 \text { if } k \text{  is odd}.\\
\end{array}
$$

{\bf Example 2.2}
$$
\begin{array}{l}
\displaystyle  [p,q]^1:=[p,q]=p \frac{\mathcal{D}^*(q)}{{\rm ord}(q)}-q \frac{\mathcal{D}^*(p)}{{\rm ord}(p)},  \text{ -- the  semi-jacobian of  } p \text{  and  }  q, \\ \\
\displaystyle [p,q]^2=p \frac{(\mathcal{D}^*)^2(q)}{[{\rm ord}(q)]_2}- 2 \frac{\mathcal{D}^*(p)}{[{\rm ord}(p)]} \frac{\mathcal{D}^*(q)}{[{\rm ord}(q)]} +q \frac{(\mathcal{D}^*)^2(p)}{[{\rm ord}(p)]_2}, \text{  -- the  semi-hessian of  }  p \text{  and  }  q .
\end{array}
$$
Up to a constant factor, the semi-hessian of the semi-covariant   $a_0$  equals
  $$ \frac{1}{2} [a_0,a_0]^2=a_0a_2-a_1^2= \begin{vmatrix}  a_0 & a_1 \\ a_1 & a_2\end{vmatrix},$$

\begin{oz}  
$$
\begin{array}{ll}
{\bf 1.}  & \text {The homogeneous polynomial   } F  \text { is called isobaric if it is an eigenvector } \\ 
&  \text { of the operator  } E,  \text{  i.e.  }  E(F)=\omega(F) F,  \omega(F) \in \mathbb{Z},
\\
{\bf 2.}  & \text {The  corresponding  eigenvalue }  \omega(F) \text{  is called the weight of the isobaric polynomial  } $F.$\\
\end{array}
$$
\end{oz}

The following theorem holds

\begin{te}
$$
\begin{array}{ll}
{\bf 1.}  & \omega(a_0^{k_0} a_1^{k_1}\cdots a_n^{k_n})=n (k_0+k_1+\cdots+k_n)-2(k_1+k_2+\cdots+k_n),\\

{\bf 2.}  & \text { if  }  s  \text { is a homogeneous isobaric semi-invariant then   } {\rm ord}(s)= \omega(s),\\

{\bf 3.}  & \text { if  }  p, q   \text { are  homogeneous isobaric semi-invariants then  }  \\
& \omega([p,q]^i)=\omega(p)+\omega(q)-2\,i.
\end{array}
$$
\end{te}
Throughout this paper  the  semi-invariant will mean an  isobaric homogeneous semi-invariant.

\begin{oz}
A semi-invariant  $S$ of the binary form of order  $n$ is called proper if  $\displaystyle \frac{\partial S}{\partial a_n} \neq 0.$
\end{oz}

In other words, a semi-invariant  is the  proper one of the binary form of degree $n$ if  it is not a semi-invariant of a binary form  of smaller order.

{\bf Problem.} Find all irreducible proper semi-invariants of the binary form of order $n.$ 

\section{The main theorems}

The following theorem is  crucial  for the constructions  of identities for the Appell polynomials.
\begin{te} Let  $\varphi_ {\mathcal{A}} \colon \mathbb{K}[a_0,a_1,\ldots, a_n] \to \mathbb{K}[x]$ be the substitution homomorphism
$$
\varphi_ {\mathcal{A}}(a_i)= A_i(x).
$$
Then    $$\ker^* \varphi_ {\mathcal{A}}=\mathcal{S}_n.$$
\end{te}
\begin{proof}
First we shall show that the homomorphism  $\varphi_ {\mathcal{A}}$ commutes with the derivative operator $\displaystyle \frac{d}{dx},$  i.e.
$$
\varphi_ {\mathcal{A}}( \mathcal{D}(h(a_0,a_1,\ldots,a_n)))=\frac{d}{dx}\left( \varphi_ {\mathcal{A}}(h(a_0,a_1,\ldots,a_n)) \right),
$$
for all $h(a_0,a_1,\ldots,a_n) \in \mathbb{K}[a_0,a_1,\ldots, a_n].$
The proof is by induction over  the degree of the polynomial 
  $h(a_0,a_1,\ldots,a_n).$

Show that the statement holds for  all  polynomials of degree 1:
$$
\varphi_ {\mathcal{A}}( \mathcal{D}(a_i))=\varphi_ {\mathcal{A}}(i a_{i-1})=i A_{i-1}(x)=\frac{d}{dx}A_i(x)=\frac{d}{dx}\varphi_ {\mathcal{A}}(a_i).
$$
Assume that it  holds for all  polynomials   $f \in \mathbb{K}[a_0,a_1,\ldots, a_n]$ , $\deg(f) \leqslant k.$
 $$\varphi_ {\mathcal{A}}( \mathcal{D}(f))=\frac{d}{dx}\varphi_ {\mathcal{A}}(f) .$$ Then for all $i$ we have
$$
\begin{array}{l}
\varphi_ {\mathcal{A}}( \mathcal{D}(a_i f))=\varphi_ {\mathcal{A}}( \mathcal{D}(a_i) f)+\varphi_ {\mathcal{A}}( a_i \mathcal{D}(f))=\varphi_ {\mathcal{A}}( \mathcal{D}(a_i)) \varphi_ {\mathcal{A}}(f)+\varphi_ {\mathcal{A}}( a_i) \varphi_ {\mathcal{A}}(\mathcal{D}(f))=\\ \\

\displaystyle =\frac{d}{dx}\varphi_ {\mathcal{A}}(a_i)  \varphi_ {\mathcal{A}}(f)+\varphi_ {\mathcal{A}}( a_i) \frac{d}{dx}\varphi_ {\mathcal{A}}(f) =\frac{d}{dx}\left( \varphi_ {\mathcal{A}}(a_i) \varphi_ {\mathcal{A}}(f) \right)=\frac{d}{dx}\left( \varphi_ {\mathcal{A}}(a_i f) \right).
\end{array}
$$
The linearity of the derivations  $\mathcal{D},$  $\displaystyle \frac{d}{dx}$ and the linearity of the homomorphism $\varphi_ {\mathcal{A}}$ yield that the statement holds for all polynomials of the degree $k+1.$

Thus, by induction  the  $\varphi_ {\mathcal{A}}$ commutes with the derivative  $\displaystyle \frac{d}{dx}.$

Show that  $\mathcal{S}_n \subset\ker^* \, \varphi_{\mathcal{A}}.$
For   $h(a_0,a_1,\ldots,a_n) \in \mathcal{S}_n$ we have  
$$
 \frac{d}{dx}\left(h(A_0(x),\ldots, A_n(x))\right)= \mathcal{D}\varphi_ {\mathcal{A}}(h(A_0(x),\ldots, A_n(x)))=\mathcal{D}(h(a_0,\ldots, a_n))=0.
$$
Therefore, $h(A_0(x),\ldots, A_n(x))$  is a constant as claimed. 

On the contrary, assume  $g(A_0(x),\ldots, A_n(x)) \in \mathbb{K}.$ Then 
$$
 \mathcal{D}(g(a_0,\ldots, a_n))=  \frac{d}{dx}g(A_0(x),\ldots, A_n(x))=0.
$$
Thus $g(a_0,\ldots, a_n) \in \mathcal{S}_n$ and  $\mathcal{S}_n= \ker^* \, \varphi_{\mathcal{A}}.$
\end{proof}

So, any semi-invariant  $S(a_0,\ldots, a_n)$ defines the identity 

$$
S(\mathcal{A})=\| S(\mathcal{A}) \|,
$$
for the Appell polynomials  $\{ A_n(x)\}.$

\begin{oz}
For the semi-invariant $S(a_0,\ldots, a_n)$ the number $\| S(\mathcal{A}) \|$ is said to be the norm of the semi-invariant with respect to the Appell polynomials $\mathcal{A}.$
\end{oz}

{\bf Example 3.1}  Let $\varGamma(a_0,a_1,a_2)=\frac{1}{2}[a_0,a_0]^2$ be the semi-hessian. Then  
$$
\begin{array}{l}
\displaystyle \varGamma(\mathcal{B})=B_0(x) B_2(x)-B_1(x)^2=\frac{1}{6}+{x}^{2}-x- \left( x-\frac{1}{2} \right) ^{2}=-\frac{1}{12},\\
\displaystyle \varGamma(\mathcal{E})=E_0(x) E_2(x)-E_1(x)^2={x}^{2}-x- \left( x-\frac{1}{2}\right) ^{2}=-\frac{1}{4},\\

\displaystyle \varGamma(\mathcal{H})=H_0(x) H_2(x)-H_1(x)^2=-1+x^2-x^2=-1,\\
\varGamma(\mathcal{T})=0.
\end{array}
 $$

\begin{te}
The semi-invariant   $S(a_0,a_1,\ldots,a_n)$ determines the identity $S(1,1,\ldots,1)=0.$
\end{te}
\begin{proof}
It it easy to see  that for a homogeneous isobaric polynomial  $S(a_0,a_1,\ldots,a_n)$  we have 
$$
S(\mathcal{T})=S(1,x,x^2,\ldots,x^n)=x^m S(1,1,\ldots,1),
$$
for some integer  $m.$ 

Therefore, $\|S(\mathcal{T})\|=0.$
From another hand, the identity  $S(\mathcal{T})=\|S(\mathcal{T})\|,$ implies  $$x^m S(1,1,\ldots,1)=0,$$
for all  $x.$ Thus  $S(1,1,\ldots,1)=0.$
\end{proof}
For the algebra of  joint semi-invariants  one can easily formulate and prove similar theorems.

\section{Identities for unique  Appel  sequence}

To describe   identities for  Appell polynomials of the same  type let us describe the low degree proper semi-invariants for the finary form  $\alpha(X,Y).$
The formula  (\ref{main}) generates  semi-invariants of degree 2, namely $[a_0,a_0]^i,$ $i=0,\ldots,n. $  Therefore, a  proper semi-invariant of the degree 2 is  the semi-transvectant 
$$
[a_0,a_0]^n=\sum_{i=0}^n (-1)^i {n \choose i} a_i a_{n-i}.
$$
Denote it by  ${\rm Dv}_n(a_0)$ and its image  $\varphi_{\mathcal{A}}({\rm Dv}_n(a_0))$ denote by  ${\rm Dv}_n(\mathcal{A}):$ $$
{\rm Dv}_n(\mathcal{A}):=\sum_{i=0}^n (-1)^i {n \choose i} A_i(x) A_{n-i}(x).
$$
It is easy to check that the variable $a_i$ can be expressed by $a_0$ as follows $\displaystyle a_i=\frac{(\mathcal{D}^*)^i(a_0)}{[n]_i}.$ So, for simplicity of notation we  write ${\rm Dv}_n(a_0)$ instead of ${\rm Dv}_n(a_0,a_1,\ldots,a_n).$

{\bf Example  4.1.} For the Bernoulli polynomials we have
$$
{\rm Dv}_n(\mathcal{B})=\sum_{i=0}^n (-1)^i {n \choose i} B_i(x) B_{n-i}(x).
$$
Now
$$
\|{\rm Dv}_n(\mathcal{B})\|=\sum_{i=0}^n (-1)^i {n \choose i} B_i(0) B_{n-i}(0)=\sum_{i=0}^n (-1)^i {n \choose i} B_i B_{n-i}.$$
From other hand, by direct calculation  one can show  \cite{Lis} that 
$$
\sum_{i=0}^n (-1)^i {n \choose i} B_i B_{n-i}=(1-n)B_n.
$$
Therefore we obtain the identity for the Bernoulli polynomials $$
\sum_{i=0}^n (-1)^i {n \choose i} B_i(x) B_{n-i}(x)=(1-n)B_n.
$$

For the Euler polynomials we propose the conjecture

{\bf Conjecture.} $
\displaystyle \sum_{i=0}^n (-1)^i {n \choose i} E_i(x) E_{n-i}(x)=-2 E_{n+1}.
$

Theorem 3.2 implies the wel-known binomial identity $\displaystyle \sum_{i=1}^{n} (-1)^i {n \choose i}=0.$ 

 In the paper  \cite{Bed_In7} we  found another proper semi-invariant of degree 2:
$$
W_n(a_0):=\sum_{i=1}^{n} (-1)^i {n \choose i} a_{n-i} a_1^i.
$$

To construct a proper semi-invariant of degree 3 use the semi-hessian $\displaystyle \frac{1}{2} [a_0,a_0]^2= \begin{vmatrix}  a_0 & a_1 \\ a_1 & a_2    \end{vmatrix} $ of the semi-invariant $a_0.$
Denote $\displaystyle {\rm Tr}_n(a_0):=[a_0,\frac{1}{2} [a_0,a_0]^2]^n.$
 \begin{te}  For  $n\geqslant 4$ the formula holds

\begin{gather*}
{\rm Tr}_n(a_0): =\sum_{i=0}^n \sum_{j=0}^{i} \frac{ (-1)^i}{[2n-4]_{i}} { n \choose i } {i\choose j}a_{n-i}  \begin{vmatrix}  [n]_j \, a_j & [n-1]_{i-j}\, a_{i-j+1} \\ [n-1]_{j} \, a_{j+1} &[n-2]_{i-j} \,a_{i-j+2}   \end{vmatrix}.
\end{gather*}
\end{te}
\begin{proof}
Since the semi-hessian  has the   weight  $2n-4,$ then  the semi-transvectant  \\ $\displaystyle [a_0,\frac{1}{2} [a_0,a_0]^2]^n$ has the order $n+2n-4-2n=n-4.$ Thus it  well-defined for $n\geqslant 4.$
We  have 
$$
(\mathcal{D}^*)^i(a_k)=[n-k]_{i}\, a_{i+k}.
$$
By the determinant derivative rule we have 
\begin{gather*}
(\mathcal{D}^*)^i \left( \begin{vmatrix}  a_0 & a_1 \\ a_1 & a_2    \end{vmatrix}\right)=\sum_{j=0}^{i} {i\choose j}\begin{vmatrix}  (\mathcal{D}^*)^j(a_0) & (\mathcal{D}^*)^{i-j}(a_1) \\ ( \mathcal{D}^*)^j(a_1) & (\mathcal{D}^*)^{i-j}(a_2)
 \end{vmatrix}
=\\
=\sum_{j=0}^{i} {i\choose j}\begin{vmatrix}  [n]_j \, a_j & [n-1]_{i-j}\, a_{i-j+1} \\ [n-1]_{j} \, a_{j+1} &[n-2]_{i-j} \,a_{i-j+2}   \end{vmatrix}.
\end{gather*}

By (3) we get 
\begin{gather*}
[a_0,[a_0,a_0]^2]^n=\sum_{i=0}^n \frac{ (-1)^i}{[n]_{n-i} [2n-4]_{i}} { n \choose i } (\mathcal{D}^*)^{n-i}(a_0)(\mathcal{D}^*)^i\left(\begin{vmatrix}  a_0 & a_1 \\ a_1 & a_2    \end{vmatrix}\right)=\\
=\sum_{i=0}^n \frac{ (-1)^i}{[2n-4]_{i}} { n \choose i }\sum_{j=0}^{i} {i\choose j}a_{n-i}  \begin{vmatrix}  [n]_j \, a_j & [n-1]_{i-j}\, a_{i-j+1} \\ [n-1]_{j} \, a_{j+1} &[n-2]_{i-j} \,a_{i-j+2}   \end{vmatrix}.
\end{gather*}
\end{proof}
As above, the direct  calculations the $n$-th semi-transvectant ($n\geqslant 4$) of the two semi-hessians yields  the semi-invariant of degree 4:

\begin{gather*}
{\rm Ch}_n(a_0):=\left [\frac{1}{2} [a_0,a_0]^2,\frac{1}{2} [a_0,a_0]^2\right ]^n=\sum_{i=0}^n \sum_{j=0}^{i} \sum_{k=0}^{n-i} \frac{ (-1)^i { n \choose i } {i\choose j} {n-i\choose k}}{[2n-4]_{i}\,[2n-4]_{n-i}} A_{i,j,k}.
\end{gather*}
where 
\begin{gather*}
A_{i,j,k}:=\begin{vmatrix}  [n]_k \, a_k & [n-1]_{n-i-k}\, a_{n-i-k+1} \\ [n-1]_{k} \, a_{k+1} &[n-2]_{n-i-k} \,a_{n-i-k+2}   \end{vmatrix} \cdot  \begin{vmatrix}  [n]_j \, a_j & [n-1]_{i-j}\, a_{i-j+1} \\ [n-1]_{j} \, a_{j+1} &[n-2]_{i-j} \,a_{i-j+2}
 \end{vmatrix}.
\end{gather*}

Now, let us consider the discriminant  of a binary form. The discriminant is a well known  invariant which can be defined as   the $(2n-1) \times (2n-1)$ determinant  of the Sylvester matrix of the binary form $\alpha(X,Y):$  
\begin{gather*}
{\rm  Discr}_n(a_0):=\begin{vmatrix}
a_0 & n a_1 & \cdots  & a_n & 0 & \cdots &\cdots  & 0 \\
0 & a_0  & \cdots  &n a_{n-1} & a_n &0& \cdots & 0 \\
\hdotsfor{8}\\
0 & \cdots  & 0  & a_{0} &n  a_1 & {n \choose 2} a_2& \cdots & a_n \\
n a_0 & (n-1) n a_1 &  \dots   & n a_{n-1} & 0&0& \cdots & 0\\
0&n a_0  & \cdots & \dots   & n a_{n-1} & 0& \cdots & 0\\
\hdotsfor{8}\\
0& 0& \cdots &n a_0 & (n-1) n a_1 & (n-2) {n \choose 2}  a_2 & \dots   & n a_{n-1} 
\end{vmatrix}
\end{gather*}
The corresponding identities has  the form
 ${\rm Discr}_{n}(\mathcal{A})=\|{\rm Discr}_n(\mathcal{A}) \|.$

{\bf Example 4.2} Let  $\mathcal{A}=\mathcal{B},$ $n=3.$ Simplifying the identity  
$$
\begin{array}{l}
{\rm Discr}_3(\mathcal{B})=\|{\rm Discr}_3(\mathcal{B})\|, 
\end{array}
$$
we obtain
$$
\begin{array}{l}
\displaystyle -27\,{B_{{3}}}(x)^{2}{B_{{0}}}(x)^{2}+162\,B_{{3}}(x)B_{{0}}(x)B_{{1}}(x)B_{{2}}(x)+81\,
{B_{{2}}}(x)^{2}{B_{{1}}}(x)^{2}-\\
\\
\displaystyle -108\,{B_{{2}}}(x)^{3}B_{{0}}(x)-108\,{B_{{1}}}(x)^{3
}B_{{3}}(x)=\frac{1}{16}.
\end{array}
$$

{\bf Conjecture.} $ \|{\rm Discr}_{n}(\mathcal{H}) \| =\prod_{k=1}^n k^k.$

Thus, we get  the  five types of identities for the Appell polynomials.

\begin{te}
Let  $\mathcal{A} =\{A_n(x) \}$ be the Appell polynomials.  Then the following identities hold
\begin{gather}
{\rm Dv}_n(\mathcal{A})=\|{\rm Dv}_n(\mathcal{A}) \|, \\
{\rm Tr}_n(\mathcal{A})=\|{\rm Tr}_n(\mathcal{A}) \|, \\
 {\rm Ch}_n(\mathcal{A})=\|{\rm Ch}_n(\mathcal{A}) \|, \\
{\rm Discr}_{n}(\mathcal{A})=\|{\rm Discr}_n(\mathcal{A}) \|,\\
W_n( \mathcal{A})=\|W_n( \mathcal{A}) \|.
 \end{gather}
\end{te}

By applying Theorem 3.3 to the above identities we derive the following binomial identities: 
\begin{gather}
\sum_{i=0}^n \sum_{j=0}^{i} \frac{ (-1)^i}{[2n-4]_{i}} { n \choose i } {i\choose j}\begin{vmatrix}  [n]_j  & [n-1]_{i-j} \\ [n-1]_{j}  &[n-2]_{i-j} \end{vmatrix}=0,\\
\sum_{i=0}^n \sum_{j=0}^{i} \sum_{k=0}^{n-i} \frac{ (-1)^i { n \choose i } {i\choose j} {n-i\choose k}}{[2n-4]_{i}\,[2n-4]_{n-i}}\begin{vmatrix}  [n]_k & [n-1]_{n-i-k} \\ [n-1]_{k} &[n-2]_{n-i-k}    \end{vmatrix} \cdot  \begin{vmatrix}  [n]_j  & [n-1]_{i-j} \\ [n-1]_{j}&[n-2]_{i-j}
 \end{vmatrix}
=0.
\end{gather}

\section{Joint identities}

Let us find  joint proper semi-invariants of the binary forms  $\alpha(X,Y)$ and  $\beta(X,Y).$

 First of all we consider the  $n$-th semi-transvectant  of the semi-invariants $a_0$ and $b_0:$
$$
{\rm Dv}_n(a_0,b_0):=[a_0,b_0]^n=\sum_{i=0}^n (-1)^i {n \choose i} a_i b_{n-i}.
$$
{\bf Example 5.1} For the Bernoulli and Euler polynomials we have 
$$
{\rm Dv}_n(\mathcal{B},\mathcal{E})=\sum_{i=0}^n (-1)^i {n \choose i} B_i(x) E_{n-i}(x).
$$
By direct calculations we get  
$$
\|{\rm Dv}_1(\mathcal{B},\mathcal{E}) \|=0, \|{\rm Dv}_2(\mathcal{B},\mathcal{E}) \|=-\frac{1}{3},\|{\rm Dv}_3(\mathcal{B},\mathcal{E}) \|=0, \|{\rm Dv}_4(\mathcal{B},\mathcal{E}) \|=\frac{7}{15}.
$$
For the general case we propose the conjecture

{\bf Conjecture.} $
\displaystyle \sum_{i=0}^n (-1)^i {n \choose i} B_i(x) E_{n-i}(x)=-2(2^{2n-1}-1)B_{2n}.
$

Similarly, the identity $$
{\rm Dv}_n(\mathcal{B},\mathcal{T})=\|{\rm Dv}_n(\mathcal{B},\mathcal{T})\|,
$$
implies that  ${\rm Dv}_n(\mathcal{B},\mathcal{T})=\displaystyle \sum_{i=0}^n (-1)^i {n \choose i} B_i(x)x^{n-i}.$ It follows $${\rm Dv}_n(\mathcal{B},\mathcal{T})_{{ |}_{x=0}}=\|{\rm Dv}_n(\mathcal{B},\mathcal{T})\|=B_n(0)=B_n.$$
After simplification we get the identity for the Bernoulli polynomials
$$
B_n(x)=\sum_{i=0}^{n-1} (-1)^{i+1} {n \choose i} B_i(x) x^{n-i}+B_n.$$
In the same way we get the identities for the Euler and Hermite polynomials
\begin{gather*}
E_n(x)=\sum_{i=0}^{n-1} (-1)^{i+1} {n \choose i} E_i(x) x^{n-i}+E_n,\\
H_n(x)=\sum_{i=0}^{n-1} (-1)^{i+1} {n \choose i} H_i(x) x^{n-i}+H_n(0).
\end{gather*}

The proper joint semi-invariants of degree 3 are the following semi-transvectants  $[a_0,[a_0,b_0]^i]^n$ for  $2i\leqslant n.$ By direct calculations for $i=1$  we get  
\begin{gather*}
{\rm Tr}_n(a_0,b_0):=[a_0,[a_0,b_0]^1]^n=\left [a_0,\begin{vmatrix}  a_0 & b_0 \\ a_1 & b_1    \end{vmatrix} \right]^n=\\
=\sum_{i=0}^n \sum_{j=0}^{i} \frac{ (-1)^i}{[2n-2]_{i}} { n \choose i } {i\choose j}a_{n-i}  \begin{vmatrix}  [n]_j \, a_j & [n]_{i-j}\, b_{i-j} \\ [n-1]_{j} \, a_{j+1} &[n-1]_{i-j} \,b_{i-j+1}   \end{vmatrix},
\end{gather*}
and 
\begin{gather*}
\overline{{\rm Tr}}_n(a_0,b_0):=[a_0,[a_0,b_0]^2]^n=\left [a_0, a_0 b_2-2 a_1 b_1+a_2 b_0 \right]^n=\\
=\sum_{i=0}^n \sum_{j=0}^{i} \frac{ (-1)^i}{[2n-2]_{i}} { n \choose i } {i\choose j}a_{n-i} A_{i,j},
\end{gather*}
where
$$
A_{i,j}:= [n]_i [n-2]_{i-j}a_i b_{i-j+2}-2 [n-1]_i [n-1]_{i-j} a_{i+1} b_{i-j+1}+[n-2]_i [n]_{i-j} a_{i+2} b_{i-j}.
$$

The well known joint covariant is the resultant of two binary form. The corresponding semi-invariant ${\rm  sRes}_n(a_0,b_0)$ has form 
\begin{gather*}
{\rm  sRes}_n(a_0,b_0):=\begin{vmatrix}
a_0 & n a_1 & \cdots  & a_n & 0 & \cdots &\cdots  & 0 \\
0 & a_0  & \cdots  & n a_{n-1} & a_n &0& \cdots & 0 \\
\hdotsfor{8}\\
0 & \cdots  & 0  & a_{0} & n a_1 & {n \choose 2} a_2& \cdots & a_n \\
b_0 & n b_1 & \cdots  & b_n & 0 & \cdots &\cdots  & 0 \\
0 & b_0  & \cdots  &n  b_{n-1} & b_n &0& \cdots & 0 \\
\hdotsfor{8}\\
0 & \cdots  & 0  & b_{0} & b_1 & {n \choose 2}  b_2& \cdots & b_n \\
\end{vmatrix}
\end{gather*}

{\bf Example 5.2} The semi-resultant of two binary forms of order $2$ leads to the following identity for the Bernoulli and Euler polynomials  
\begin{gather*}
{\rm sRes}_2(\mathcal{B},\mathcal{E})=\|{\rm sRes}_2(\mathcal{B},\mathcal{E})\|.
\end{gather*}
We expand the determinants and get 
$$
\begin{array}{l}
\displaystyle {B_{{2}}}(x)^{2}{E_{{0}}}(x)^{2}-2\,B_{{2}}(x)E_{{0}}(x)E_{{2}}(x)B_{{0}}(x)+{E_{{2}}}(x)^{
2}{B_{{0}}}(x)^{2}-4\,B_{{1}}(x)B_{{2}}(x)E_{{1}}(x)E_{{0}}(x)-\\
\displaystyle -4\,B_{{1}}(x)E_{{1}}(x)E_{{2
}}(x)B_{{0}}(x)+4\,E_{{2}}(x){B_{{1}}}(x)^{2}E_{{0}}(x)+4\,B_{{0}}(x) B_{{2}}(x){E_{{1}}}^{2
}(x)=\frac{1}{36}.
\end{array}
$$
Let us  find out the  joint proper semi-invariants of the binary forms  $\alpha(X,Y),$   $\beta(X,Y)$ and $\gamma(X,Y).$
Since the semi-jacobian   $[b_0,c_0]$ has the weight  $2\,n-2,$ the semi-transvectant  $[a_0,[b_0,c_0]]^n$ is well-defined. We have
$$
(\mathcal{D}^*)^i(b_k)=[n-k]_{i}\, b_{i+k}, (\mathcal{D}^*)^i(c_k)=[n-k]_{i}\, c_{i+k}.
$$
Therefore
\begin{gather*}
(\mathcal{D}^*)^i\left( \begin{vmatrix}  b_0 & c_0 \\ b_1 & c_1    \end{vmatrix}\right)=\sum_{j=0}^{i} {i\choose j} [n]_{j} [n-1]_{i-j}\begin{vmatrix} b_j & c_j \\ b_{i-j+1} &c_{i-j+1}  \end{vmatrix}.
\end{gather*}
Thus 
\begin{gather*}
[a_0,[b_0,c_0]]^n=\sum_{i=0}^n \frac{ (-1)^i}{[n]_{n-i} [2n-2]_{i}} { n \choose i } (\mathcal{D}^*)^{n-i}(a_0)(\mathcal{D}^*)^{i}\left( \begin{vmatrix} b_0 & c_0\\ b_{1} &c_{1}  \end{vmatrix}\right)=\\
=\sum_{i=0}^n \frac{ (-1)^i}{[n]_{n-i} [2n-2]_{i}} { n \choose i }[n]_{n-i}a_{n-i}\sum_{j=0}^{i} {i\choose j} [n]_{j} [n-1]_{i-j}\begin{vmatrix} b_j & c_j \\ b_{i-j+1} & c_{i-j+1}  \end{vmatrix}.
\end{gather*}
It implies
\begin{gather*}
{\rm Tr}_n(a_0,b_0,c_0):=[a_0,[b_0,c_0]]^n=\sum_{i=0}^n \sum_{j=0}^{i} \frac{ (-1)^i [n]_{j} [n-1]_{i-j}}{ [2n-2]_{i}} { n \choose i } {i\choose j}a_{n-i}  \begin{vmatrix} b_j & c_j \\ b_{i-j+1} & c_{i-j+1}  \end{vmatrix}.
\end{gather*}
 Finally, let us find  the  joint proper semi-invariants of the four  binary forms  $\alpha(X,Y),$   $\beta(X,Y),$ $\gamma(X,Y)$ and  $\delta(X,Y).$
It is easy to see that the determinant  
$$
\Delta:=\begin{vmatrix}
 b_0& c_0& d_0 \\
 b_1 & c_1& d_1 \\
 b_2 &  c_2 & d_2\\
\end{vmatrix}
$$
is a semi-invariant with the weight  $3n-6.$ Then  the semi-teransvectant $[d_0,\Delta]^n$ is well defined for  $n \geqslant 3.$
As above, we obtain 
\begin{gather*}
{\rm Ch}_n(a_0,b_0,c_0,d_0):=[d_0,\Delta]^n=\\
=\sum_{i=0}^n \frac{ (-1)^i}{[3n-6]_{i}} { n \choose i } a_{n-i}\sum_{i_1+i_2+i_3=i} \frac{i!}{ i_1! i_2! i_3!} \begin{vmatrix}
 [n]_{i_1} b_{i_1}&  [n]_{i_2} c_{i_2} & [n]_{i_3}  d_{i_3} \\
  [n-1]_{i_1}b_{i_1+1} &   [n-1]_{i_2}c_{i_2+1}&   [n-1]_{i_3}d_{i_3+1} \\
  [n-2]_{i_1}b_{i_1+2} &   [n-2]_{i_2}c_{i_2+2}&   [n-2]_{i_3}d_{i_3+2} \\
\end{vmatrix}.
\end{gather*}

Therefore we get the following identities for the   Appell polynomials of  different series $\mathcal{A}_1,$ $\mathcal{A}_2,$ $\mathcal{A}_3,$ $\mathcal{A}_4:$

\begin{te}

\begin{gather}
{\rm Dv}_n(\mathcal{A}_1,\mathcal{A}_2)=\|{\rm Dv}_n(\mathcal{A}_1,\mathcal{A}_2) \|, \\
{\rm Tr}_n(\mathcal{A}_1,\mathcal{A}_2)=\|{\rm Tr}_n(\mathcal{A}_1,\mathcal{A}_2)\|, \\
\overline{{\rm Tr}}_n(\mathcal{A}_1,\mathcal{A}_2)=\|\overline{{\rm Tr}}_n(\mathcal{A}_1,\mathcal{A}_2)\|, \\
 {\rm Ch}_n(\mathcal{A}_1,\mathcal{A}_2)=\| {\rm Ch}_n(\mathcal{A}_1,\mathcal{A}_2) \|, \\
{\rm sRes}_{n}(\mathcal{A}_1,\mathcal{A}_2)=\|{\rm  sRes}_n(\mathcal{A}_1,\mathcal{A}_2) \|,\\
 {\rm Tr}_n(\mathcal{A}_1,\mathcal{A}_2,\mathcal{A}_3)=\| {\rm Tr}_n(\mathcal{A}_1,\mathcal{A}_2,\mathcal{A}_3) \|,\\
{\rm Ch}_n(\mathcal{A}_1,\mathcal{A}_2,\mathcal{A}_3,\mathcal{A}_4)=\| {\rm Ch}_n(\mathcal{A}_1,\mathcal{A}_2,\mathcal{A}_3,\mathcal{A}_4) \|.
 \end{gather}
\end{te}

By using Theorem  3.3 we get the binomial identities:
\begin{gather}
\sum_{i=0}^n \sum_{j=0}^{i} \frac{ (-1)^i}{[2n-2]_{i}} { n \choose i } {i\choose j} \begin{vmatrix}  [n]_j  & [n]_{i-j}\\
 [n-1]_{j}&[n-1]_{i-j}  \end{vmatrix}=0,\\
\sum_{i=0}^n \sum_{j=0}^{i} \frac{ (-1)^i}{[2n-2]_{i}} { n \choose i } {i\choose j}([n]_i [n-2]_{i-j}-2 [n-1]_i [n-1]_{i-j}+[n-2]_i [n]_{i-j} )=0,\\
\sum_{i=0}^n \frac{ (-1)^i}{[3n-6]_{i}} { n \choose i } \sum_{i_1+i_2+i_3=i} \frac{i!}{ i_1! i_2! i_3!} \begin{vmatrix}
 [n]_{i_1} &  [n]_{i_2} & [n]_{i_3}  \\
  [n-1]_{i_1} &   [n-1]_{i_2}&   [n-1]_{i_3} \\
  [n-2]_{i_1} &   [n-2]_{i_2}&   [n-2]_{i_3} \\
\end{vmatrix}=0.
\end{gather}
Unfortunatelly, all efforts to find the norms of the above semi-invariants  were unsuccessful.

\end{document}